\documentclass{amsart}

\usepackage[utf8]{inputenc}
\usepackage{amsmath, amssymb, amsthm}
\usepackage{mathrsfs}
\usepackage{enumerate}

\newtheorem{thm}{Theorem}[section]

\newtheorem{cor}[thm]{Corollary}
\newtheorem{lem}[thm]{Lemma}
\theoremstyle{definition}
\newtheorem{defi}[thm]{Definition}

\newtheorem*{case1}{Case 1}
\newtheorem*{case2}{Case 2}

\newcommand{\N}{\mathbb{N}}
\newcommand{\R}{\mathbb{R}}
\newcommand{\X}{\mathcal{X}}
\newcommand{\varep}{\varepsilon}
\newcommand{\U}{\mathcal{U}}
\newcommand{\V}{\mathcal{V}}
\DeclareMathOperator{\fin}{FIN}
\DeclareMathOperator{\supp}{supp}
\DeclareMathOperator{\stem}{stem}
\DeclareMathOperator{\osc}{osc}

\begin{document}

\title{An infinite-dimensional version of Gowers' $\mathrm{FIN}_{\pm k}$ theorem}
\author[J. K. Kawach]{Jamal K. Kawach}
\date{\today}
\address{Department of Mathematics\\ University of Toronto\\ Toronto, Ontario, M5S 2E4, Canada.}
\email{jamal.kawach@mail.utoronto.ca}

\begin{abstract}
We prove an infinite-dimensional version of an approximate Ramsey theorem of Gowers, initially used to show that every Lipschitz function on the unit sphere of $c_0$ is oscillation stable. To do so, we use the theory of ultra-Ramsey spaces developed by Todorcevic in order to obtain an Ellentuck-type theorem for the space of all infinite block sequences in $\mathrm{FIN}_{\pm k}$.
\end{abstract}

\maketitle

\section{Introduction}
Let $X$ be a Banach space and let $S_X$ be its unit sphere. A function $f : S_X \rightarrow \R$ is \emph{oscillation stable} if for every $\varep > 0$ and every closed infinite-dimensional subspace $Y$ of $X$ there is a closed infinite-dimensional subspace $Z$ of $Y$ such that $$\osc(f, S_Z) := \sup\{|f(x) - f(y)| : x,y \in S_Z\} < \varep.$$ Gowers' $c_0$ theorem, originally proved in \cite{G}, states that every Lipschitz (or, more generally, uniformly continuous) function $f : S_{c_0} \rightarrow \R$ is oscillation stable. The proof of this theorem relies on a Ramsey-type result about the space of all finitely-supported functions $p : \omega \rightarrow \{0, \pm 1, \dots, \pm k\}$ which take at least one of the values $\pm k$. The main goal of this note is to extend this latter result to its natural infinite-dimensional analogue (Theorem \ref{mainthm} below).

Before we can state these results, we fix some notation. Let $\omega$ denote the set of all non-negative integers, and $\N$ the set of all positive integers. We will often identify each ordinal $m < \omega$ with the set $\{0, \dots, m-1\}$ of its predecessors. Given $k \in \N$, let $\fin_{\pm k}$ denote the set of all functions $p: \omega \rightarrow \{0, \pm 1, \dots, \pm k\}$ such that $$\supp p := \{n < \omega : p(n) \neq 0\}$$ is finite and such that $p$ achieves at least one of the values $\pm k$. Given $p,q \in \fin_{\pm k}$ we write $p < q$ whenever $\max \supp p < \min \supp q$. In this case we will write $p+ q$ for the element of $\fin_{\pm k}$ given by the coordinate-wise sum of $p$ and $q$. This operation gives $\fin_{\pm k}$ the structure of a partial semigroup.

We also have an operation between various $\fin$ spaces: The \emph{tetris operation} $T : \fin_{\pm k} \rightarrow \fin_{\pm (k - 1)}$ is defined by
\[ T(p)(n) := \begin{cases}
	p(n) - 1 & \text{ if $p(n) > 0$},\\
	0 & \text{ if $p(n) = 0$},\\
	p(n) + 1 & \text{ if $p(n) < 0$}.
	\end{cases}
\]
(The above terminology was not used by Gowers in \cite{G} but was introduced by Todorcevic in \cite{T} and has since become standard.) It is easy to check that $T$ is a surjective homomorphism of partial semigroups. For $\alpha \leq \omega$, a sequence $(p_n)_{n < \alpha}$ is a \emph{block sequence in $\fin_{\pm k}$} if $p_n \in \fin_{\pm k}$ and $p_n < p_m$ for all $n < m < \alpha$. Given a block sequence $P = (p_n)_{n < \alpha}$ in $\fin_{\pm k}$, the \emph{partial subsemigroup of $\fin_{\pm k}$ generated by $P$} is defined as
\begin{equation*}
\begin{split}
\langle P \rangle_{\pm k} := & \{\varep_0 T^{j_0} (p_{n_0}) + \dots + \varep_m T^{j_m} (p_{n_m}) : m < \omega, n_0 < \dots < n_m < \alpha, \\
& \varep_0, \dots, \varep_m \in \{\pm 1\},  j_0, \dots, j_m < k \text{ and $\min j_i = 0$}\}.
\end{split}
\end{equation*} If $Q = (q_n)_{n < \beta}$, $\beta \leq \alpha$ is another block sequence, we write $Q \leq P$ and say $Q$ is a \emph{block subsequence of $P$} whenever $q_n \in \langle P \rangle_{\pm k}$ for all $n < \beta$.

We will work exclusively with the $\ell_\infty$ norm given by $$||p|| :=  \sup_{n < \omega} |p(n)|$$ where $p \in \fin_{\pm k}$ and $k \in \N$. For a subset $A \subseteq \fin_{\pm k}$ and $\varep > 0$, define $$(A)_\varep := \{p \in \fin_{\pm k} : (\exists q\in A) \, ||p - q|| \leq \varep\}.$$ We can now state the following theorem of Gowers, originally proved in \cite{G} using the theory of idempotent ultrafilters in order to show that every real-valued Lipschitz function on $S_{c_0}$ is oscillation stable (see \cite{AT,K,T} for other proofs).

\begin{thm}[Gowers]\label{gowersthm} For every $k, r \in \N$ and every $c: \fin_{\pm k} \rightarrow r$ there is $i < r$ such that $\left(c^{-1}\{i\}\right)_1$ contains a partial subsemigroup of $\fin_{\pm k}$ generated by an infinite block sequence. \end{thm}

It is worth mentioning here that while Gowers' theorem is an approximate Ramsey-theoretic result, there is an exact version (also proved in \cite{G}) for the spaces $\fin_k$ consisting of all finitely-supported functions $p : \omega \rightarrow k + 1$ which achieve the value $k$. This latter result acts as a pigeonhole principle and can be used via the framework of topological Ramsey spaces as in \cite{T} to prove an infinite-dimensional version for the space $\fin_ k^{[\infty]}$ of all infinite block sequences in $\fin_k$, thus generalizing a result of Milliken \cite{M} corresponding to the case $k=1$ (which in turn corresponds to the infinite-dimensional version of Hindman's theorem \cite{H}). Since Theorem \ref{gowersthm} is not an exact Ramsey-theoretic result, it cannot be used directly to prove an infinite-dimensional analogue using the theory of topological Ramsey spaces developed in \cite{T}. Our goal is to show that such an analogue can still be obtained even though there is no pigeonhole principle for $\fin_{\pm k}$.

We will work with multi-dimensional versions of the $\fin_{\pm k}$ spaces defined above. For each $m \in \N$, let $\fin_{\pm k}^{[m]}$ be the set of all block sequences in $\fin_{\pm k}$ of length $m$. We also let $$\fin_{\pm k}^{[< \infty]} := \bigcup_{m \in \N} \fin_{\pm k}^{[m]}$$ be the set of all finite block sequences in $\fin_{\pm k}$. Furthermore, let $\fin_{\pm k}^{[\infty]}$ denote the set of all infinite block sequences in $\fin_{\pm k}$. For each $\alpha \in \N \cup \{\infty\}$ we extend the $\ell_\infty$ norm to a metric on $\fin_{\pm k}^{[\alpha]}$ by setting, for $P = (p_n)_{n < \alpha}$ and $Q = (q_n)_{n<\alpha}$, $$||P-Q|| := \sup_{n < \alpha} ||p_n - q_n||.$$ Finally, for $\alpha  \in \N \cup \{\infty\}$, $\X \subseteq \fin_{\pm k}^{[\alpha]}$ and $\varep > 0$, define $$(\X)_\varep := \{P \in \fin_{\pm k}^{[\alpha]} : (\exists Q\in \X) \,||P-Q|| \leq \varep\}.$$

It is well-known that infinite-dimensional Ramsey-theoretic results do not hold in general for all colourings. To obtain positive results, a topological restriction on the permitted colourings is needed. In our case we work with the \emph{metrizable topology} on $\fin_{\pm k}^{[\infty]}$ which is generated by basic open sets of the form $$[(q_0, \dots, q_{m-1})]  := \{(p_n)_{n < \omega} \in \fin_{\pm k}^{[\infty]} :  q_i = p_i \text{ for all $i < m$}\}$$ where $m < \omega$ and $(q_0, \dots, q_{m-1}) \in \fin_{\pm k}^{[m]}$. This is the topology inherited by $\fin_{\pm k}^{[\infty]}$ when viewed as a subspace of the Tychonov product $\left(\fin_{\pm k}^{[<\infty]}\right)^\omega$ via the natural mapping $$ P = (p_n)_{n< \omega} \mapsto (r_n(P))_{n < \omega}$$ where $r_n(P) := (p_i)_{i < n}$, and where $\fin_{\pm k}^{[<\infty]}$ is given the discrete topology.

We now describe the topological restriction mentioned above. First recall that a \emph{Souslin scheme} is a family of sets $(X_s)_{s \in \omega^{< \omega}}$ indexed by finite sequences of non-negative integers. The \emph{Souslin operation} turns a Souslin scheme $(X_s)_{s \in \omega^{< \omega}}$ into the set $$\bigcup_{x \in \omega^\omega} \bigcap_{n < \omega} X_{x \restriction n}$$ where $\omega^\omega$ denotes the set of all infinite sequences in $\omega$. Given a topological space $X$, the field of \emph{Souslin measurable} sets is the smallest field of subsets of $X$ which contains all open subsets of $X$ and is closed under the Souslin operation. In particular, every analytic (and hence Borel) subset of $X$ is Souslin measurable (see, e.g., \cite[Section 25.C]{Ke}). Finally, a colouring $c : X \rightarrow r$ is Souslin measurable if $c^{-1}\{i\}$ is Souslin measurable for each $i < r$.

Let $\langle P \rangle_{\pm k}^{[\infty]}$ denote the set of all $Q \in \fin_{\pm k}^{[\infty]}$ such that $Q \leq P$. The purpose of this note is to extend Gowers' $\fin_{\pm k}$ theorem to the following analogue for $\fin_{\pm k}^{[\infty]}$. The proof will involve a synthesis of techniques introduced by Todorcevic in \cite{T} and Kanellopoulos in \cite{K}.

\begin{thm}\label{mainthm} For every $k, r \in \N$ and every Souslin measurable $c : \fin_{\pm k}^{[\infty]} \rightarrow r$ there are $i < r$ and an infinite block sequence $P \in \fin_{\pm k}^{[\infty]}$ such that $$\langle P \rangle_{\pm k}^{[\infty]} \subseteq \left(c^{-1}\{i\}\right)_1.$$ \end{thm}

The rest of this paper is organized as follows. In Section 2 we follow the approach taken in \cite[Chapter 7]{T} and review the theory of $\U$-trees (originally introduced by Blass in \cite{B}) and the $\U$-topology, which refines the metrizable topology and allows for an Ellentuck-type theorem without the need for a pigeonhole principle. In Section 3 we define a subclass of $\U$-trees which are closed under a tetris-like operation and prove a lemma which says that, up to a fixed error, any $\U$-tree can be enlarged so that it becomes closed under such an operation. We then use this lemma in Section 4 to prove Theorem \ref{mainthm} and obtain some standard corollaries.

\subsection*{Acknowledgements} The author would like to thank Professor Stevo Todorcevic for his guidance and for suggesting the problem addressed in this paper. The author is also grateful to Professor Jordi Lopez-Abad for his support.

\section{An ultra-Ramsey space of infinite block sequences in $\fin_{\pm k}$}
In the setting of ultra-Ramsey theory, we work with a special class of trees of countably infinite height which branch according to a given ultrafilter. Recall that an \emph{ultrafilter} on a set $X$ is a collection $\U$ of subsets of $X$ satisfying the following four properties:
\begin{enumerate}
	\item $\varnothing \not \in \U$.
	\item $A, B \in \U$ implies $A \cap B \in \U$.
	\item $A \in \U, B \supseteq A$ implies $B \in \U$.
	\item For every $A \subseteq X$, either $A \in \U$ or $X \setminus A \in \U$.
\end{enumerate}
Let $\beta X$ denote the set of all ultrafilters on $X$; then $\beta X$ is a compact Hausdorff space under the topology generated by basic open sets of the form $$\overline{A} := \{\U \in \beta X : A \in \U\}$$ where $A$ is a non-empty subset of $X$. It is useful to view ultrafilters as quantifiers (e.g. as in Blass \cite{B1}) in the following way. Let $\U$ be an ultrafilter on a set $X$. Given a first-order formula $\varphi(x)$ with a free variable $x$ ranging over elements of $X$, we write $$(\U x) \varphi(x) \iff \{x \in X : \varphi(x)\} \in \U.$$ Using the ultrafilter properties above it is easy to check that ultrafilter quantifiers commute with conjunction and negation of first-order formulas, i.e. we have $$(\U x)\varphi(x) \wedge (\U x)\psi(x) \iff (\U x)(\varphi(x) \wedge \psi(x)),$$ $$\neg (\U x)(\varphi(x)) \iff (\U x)(\neg \varphi(x))$$ for any first-order formulas $\varphi(x)$ and $\psi(x)$.

We will primarily be concerned with ultrafilters on $\fin_{\pm k}$. Given two ultrafilters $\U, \V \in \beta \fin_{\pm k}$, define the \emph{sum} of $\U$ and $\V$ by declaring $$A \in \U + \V \iff (\U p)(\V q)\left(p+q \in A\right)$$ for $A \subseteq \fin_{\pm k}$. To ensure that this operation is always defined we restrict our attention to the set of all \emph{cofinite} ultrafilters on $\fin_{\pm k}$, i.e. ultrafilters $\U \in \beta \fin_{\pm k}$ which satisfy $$X_m := \{p \in \fin_{\pm k} : p(n) = 0 \text{ for all $n < m$}\} \in \U$$ for all $m < \omega$. Let $\gamma \fin_{\pm k}$ denote the set of all cofinite $\U \in \beta \fin_{\pm k}$. Then $(\gamma \fin_{\pm k}, +)$ is a compact semigroup. (We refer the reader to \cite[Chapter 2]{T} for details.) We also extend the tetris operation $T : \fin_{\pm k} \rightarrow \fin_{\pm (k-1)}$ to a map $T : \gamma \fin_{\pm k} \rightarrow \gamma \fin_{\pm (k-1)}$ by setting $$A \in T(\U) \iff T^{-1}(A) \in \U$$ for each $A \subseteq \fin_{\pm (k-1)}$. This extension is a continuous surjective homomorphism.  Below we will consider the sign-flipped version of the tetris operation given by $$-T : \fin_{\pm k} \rightarrow \fin_{\pm (k-1)} : p \mapsto -T(p)$$ together with its extension to $\gamma \fin_{\pm k}$ (the definition of which is analogous to the extension of $T$ to $\gamma \fin_{\pm k}$ above).

Given $A \subseteq \fin_{\pm k}$, let $-A := \{-x : x \in A\}$. We will need the following result, the proof which of uses the general theory of idempotents in compact semigroups.

\begin{lem}\label{lem1} There exists a cofinite ultrafilter $\U$ on $\fin_{\pm k}$ such that $$\U + (-T)^j \U = (-T)^j \U + \U = \U \text{ for all $j \in \{0, \dots, k\}$}.$$ Furthermore, $\U$ is \emph{subsymmetric}: For every $A \in \U$ we have $-(A)_1 \in \U$. \end{lem}

The proof of the first part of the above result can be found in \cite[Chapter III.5]{AT} or \cite[Lemma 4]{K}. The second part follows from the first (see \cite[Lemma 11]{K}) but we point out here that the theory of subsymmetric ultrafilters was first developed in \cite[Chapter 2]{T} (and in the earlier manuscript \cite{T1}) and is used there to give an ultrafilter proof of Gowers' theorem. Note that the ultrafilter $\U$ given by Lemma \ref{lem1} has the property that, for any $A \in \U$ and $j < k$, $$(\U f)(\U g) \left(\{f, g, f+ (-T)^j(g), (-T)^j(f) + g\} \subseteq A\right).$$ Since ultrafilter quantifiers commute with finite conjunctions it follows that $$(\U f)(\U g) \left(\{f, g, f+ (-T)^j(g), (-T)^j(f) + g: j < k\} \subseteq A\right)$$ for any $A \in \U$.

We now proceed to describe a class of trees which form the basis for the required ultra-Ramsey theory. To this end, for each $k \in \N$ we view the space $\fin_{\pm k}^{[< \infty]}$ as a tree ordered by end-extension $\sqsubseteq$ and with root $\varnothing$, the empty sequence. Unless otherwise specified, for the rest of this paper we fix $k \in \N$ together with the ultrafilter $\U$ on $\fin_{\pm k}$ given by Lemma \ref{lem1}. The next two definitions are adapted from \cite[Chapter 7.2]{T}.

\begin{defi} A \emph{$\U$-tree} is a downward closed subtree $U \subseteq \fin_{\pm k}^{[<\infty]}$ such that $$U_t := \{p \in \fin_{\pm k} : (t, p) \in U\} \in \U$$ for all $t \in U$. The \emph{stem} of $U$, denoted $\stem(U)$, is the $\sqsubseteq$-maximal element of $U$ which is comparable to every other node of the tree.\end{defi}

Given a $\U$-tree $U$, the set of infinite branches of $U$ is denoted by $$[U] := \{(p_n)_{n < \omega} \in \fin_{\pm k}^{[\infty]} : (p_0, \dots, p_m) \in U \text{ for all $m < \omega$}\}.$$ For $t \in U$ let $|t|$ denote the \emph{length} of $t$, which is just the domain of $t$ when viewed as a finite sequence in $\fin_{\pm k}^{[<\infty]}$. For $m < \omega$, the {$m^{\rm{th}}$ level} $U(m)$ of $U$ is the set of all $t \in U$ of length $m$.

In order to prove an infinite-dimensional version of Theorem \ref{gowersthm} we work with a topology defined using $\U$-trees and which extends the usual metrizable topology on $\fin_{\pm k}^{[\infty]}$. Working in this topology allows us to remedy the fact that the space $\fin_{\pm k}$ lacks an exact pigeonhole principle.

\begin{defi} Let $\X \subseteq \fin_{\pm k}^{[\infty]}$. $\X$ is \emph{$\U$-open} if for every $A \in \X$ there is a $\U$-tree $U$ such that $A \in [U] \subseteq \X$. $\X$ is \emph{$\U$-Ramsey} if for every $\U$-tree $U$ there is a $\U$-subtree $U' \subseteq U$ with $\stem(U) = \stem(U')$ such that $[U'] \subseteq \X$ or $[U'] \subseteq \X^c$. If the second alternative always holds then we say $\X$ is \emph{$\U$-Ramsey null}. \end{defi}

The collection of all $\U$-open subsets of $\fin_{\pm k}^{[\infty]}$ forms a topology, called the \emph{$\U$-topology}, which refines the metrizable topology of $\fin_{\pm k}^{[\infty]}$. The next two results are adapted from \cite[Chapter 7.2]{T} by replacing the tree $\N^{[<\infty]}$ of finite subsets of $\N$ ordered by end-extension with the tree $\fin_{\pm k}^{[<\infty]}$. We state them in our context without proof. First, recall that a subset $A$ of a topological space $X$ has the \emph{property of Baire} if there is an open set $U \subseteq X$ such that the symmetric difference of $A$ and $U$ is meager in $X$. We then have the following version of Todorcevic's ultra-Ellentuck theorem, which builds on a theorem of Ellentuck \cite{E} relating the notions of Baire and Ramsey in the setting of $\N^{[\infty]}$, the set of all infinite subsets of $\N$.

\begin{thm} Let $\X \subseteq \fin_{\pm k}^{[\infty]}$. Then $\X$ has the property of Baire relative to the $\U$-topology if and only if $\X$ is $\U$-Ramsey. Furthermore, $\X$ is meager with respect to the $\U$-topology if and only if $\X$ is $\U$-Ramsey null. \end{thm}

The next result uses a classical fact of Nikodym (see, e.g., \cite[Chapter 4.1]{T}) which says that, in any topological space, the property of Baire is preserved under the Souslin operation.

\begin{cor}\label{cor1} For every $r \in \N$ and every Souslin measurable $c : \fin_{\pm k}^{[\infty]} \rightarrow r$ there are $i < r$ and a $\U$-tree $U$ with stem $\varnothing$ such that $[U] \subseteq c^{-1}\{i\}$. \end{cor}

\section{$S$-closed $\U$-trees}
In this brief section we define a class of subtrees which will allow us to inductively construct certain block sequences during the proof of Theorem \ref{mainthm}. First, notice that if $p, q \in \fin_{\pm k}$ satisfy $|| p - q || \leq 1$, then $$n \in (\supp p \setminus \supp q) \cup (\supp q \setminus \supp p) \implies |p(n)|, |q(n)| \leq 1.$$ This motivates the following weak version of the tetris operation: Given $p \in \fin_{\pm k}$ define $S(p) \in \fin_{\pm k}$ by
\[ S(p)(n) := \begin{cases}
	p(n) & \text{ if $|p(n)| \neq 1$}\\
	0 & \text{ if $|p(n)| = 1$}.
\end{cases}
\]
We will repeatedly use the fact that $||p - S(p)|| \leq 1$ for all $p \in \fin_{\pm k}$. In particular, notice that $||p-q|| \leq 1$ implies $\supp S(p) \subseteq \supp q$ and $||S(p)-q|| \leq 2$. This will allow us to control the supports of elements which are close to a fixed $q \in \fin_{\pm k}$. Also note that $S$ is \emph{idempotent}, i.e. $S \circ S = S$. The following lemma allows us to replace a given $\U$-tree with one which behaves well with respect to $S$, at the cost of adding an approximate constant.

\begin{lem}\label{lem2} Suppose $V$ is a $\U$-tree with $\stem(V) = \varnothing$. Then there is a $\U$-tree $U$ with $\stem(U) = \varnothing$ such that $[U] \subseteq ([V])_1$ and such that $U$ is \emph{$S$-closed}: For every $t \in U$ and every $p \in \fin_{\pm k}$, we have $$(t, p) \in U \rightarrow (t,S(p)) \in U.$$
\end{lem}
\begin{proof} Fix a well-ordering $<$ of $\fin_{\pm k}^{[<\infty]}$. We construct, by induction on $n \geq 1$, each level $U(n)$ of $U$ above $\varnothing$ together with projections $\pi_n : U(n) \rightarrow V(n)$ satisfying $||t - \pi_n(t)||\leq 1$ for all $t \in U(n)$. To begin, take $U_\varnothing := V_\varnothing \cup S(V_\varnothing)$ and hence $$U(1) := \{(p) \in \fin_{\pm k}^{[1]}: p \in U_\varnothing\}.$$ The projection $\pi_1 : U(1) \rightarrow V(1)$ is defined by setting, for $t = (p) \in U(1)$,
\[ \pi_1(t) := \begin{cases}
	(p) & \text{ if $p \in V_\varnothing$}\\
	\left( \min\left(V_\varnothing \cap S^{-1}(p)\right)\right) & \text{ otherwise}
\end{cases}
\]
where the minimum is taken with respect to $<$. Note that such a minimum exists, since if $p \in U_\varnothing \setminus V_\varnothing$ then we must have $p \in S(V_\varnothing)$ and so there is $q \in V_\varnothing$ such that $S(q) = p$. Furthermore, since $||q - S(q)|| \leq 1$ for all $q \in \fin_{\pm k}$, we have $||t - \pi_1(t)|| \leq 1$ for all $t \in U(1)$.

Now suppose we have constructed the first $m > 1$ levels $U(1), \dots, U(m)$ of $U$ above $\varnothing$ with their corresponding projections $\pi_1, \dots, \pi_m$. For each $t \in U(m)$, set $U_t := V_{\pi_m(t)} \cup S(V_{\pi_m(t)})$. We then define $$U(m+1) := \{(s,p) \in \fin_{\pm k}^{[m+1]} : s \in U(m), \, p \in U_s\}.$$ The projection $\pi_{m+1} : U(m+1) \rightarrow V(m+1)$ is defined by setting, for $t = (s, p) \in U(m+1)$ with $s \in U(m)$ and $p \in U_s$,
\[ \pi_{m+1}(t) := \begin{cases}
	(\pi_m(s), p) & \text{ if $p \in V_{\pi_m(s)}$}\\
	\left(\pi_m(s), \min\left(V_{\pi_m(s)} \cap S^{-1}(p)\right)\right) & \text{ otherwise}
\end{cases}
\]
where the minimum is taken with respect to $<$. Inductively we have $||s - \pi_m(s)|| \leq 1$ and so by definition of $S$ we have $||t - \pi_{m+1}(t)|| \leq 1$. This completes the inductive construction of $U$.

The fact that $U$ is $S$-closed follows easily from the above construction. To finish, we check that $[U] \subseteq ([V])_1$. Let $P = (p_n)_{n < \omega}$ be an infinite block sequence corresponding to a branch of $U$. We define a projection $\pi_\infty : [U] \rightarrow [V]$ by setting $$\pi_\infty(P) := (\pi_n \circ r_n(P))_{n \in \N}$$ where $r_n : [U] \rightarrow U(n)$ is the $n^{\rm{th}}$ restriction mapping given by $$r_n(P) := (p_0, \dots, p_{n-1}).$$ Note that $\pi_\infty(P)$ is indeed a branch in $V$ since $s \sqsubseteq t$ implies $\pi_{|s|}(s) \sqsubseteq \pi_{|t|}(t)$ for any $s, t \in U$. Since for every $P \in [U]$ we have $||P - \pi_\infty(P)|| \leq 1$ and $\pi_\infty(P) \in [V]$, we obtain that $[U] \subseteq ([V])_1$. \end{proof}

\section{The proof of Theorem \ref{mainthm}}
In this section we give a proof of the main theorem of this note. To do so, we first need to consider the following modification of the usual notion of block subsequence. Given a block sequence $P = (p_n)_{n < \omega} \in \fin_{\pm k}^{[\infty]}$, let $\langle P \rangle_{(-T)}$ be the partial subsemigroup consisting of all vectors of the form $$(-T)^{j_0}(p_{n_0}) + \dots + (-T)^{j_m}(p_{n_m})$$ where $m < \omega, n_0 < \dots < n_m < \omega$ and $j_0, \dots, j_m < k$ are such that $\min j_i = 0$. If $Q = (q_n)_{n < \omega}$ is another block sequence, write $Q \leq_{(-T)} P$ to denote that $q_n \in \langle P \rangle_{(-T)}$ for every $n < \omega$. We define $\langle P \rangle_{(-T)}$ for finite block sequences $P = (p_n)_{n<m}$ similarly; in this case we write $\langle p_0, \dots, p_{m-1} \rangle_{(-T)}$ for the corresponding (finite) partial subsemigroup. 

\begin{lem}\label{lem3} Let $U$ be a $\U$-tree with stem $\varnothing$. There is $P = (p_n)_{n < \omega} \in \fin_{\pm k}^{[\infty]}$ such that $Q \leq_{(-T)} P$ implies $Q \in [U]$. \end{lem}

\begin{proof} By induction on $n < \omega$ we define two sequences $A_0 \supseteq A_1 \supseteq \dots$ and $p_0 < p_1 < \dots$ such that, for all $n < \omega$,
	\begin{enumerate}
		\item $p_n \in A_n \in \U$,
		\item $A_{n+1} \subseteq \{q \in \fin_{\pm k} : \langle p_n, q \rangle_{(-T)} \subseteq A_n\}$, and
		\item $A_n \subseteq U_t \cap -(U_t)_1$ for every $t \in U$ such that $$\supp \bigcup t \subseteq \bigcup_{i < n} \supp p_i$$
	\end{enumerate}
where, for a node $t = (t_0, \dots, t_{m-1}) \in U$, $\bigcup t$ is the element $\sum_{i < m} t_i \in \fin_{\pm k}$. To start, take $A_0 := U_\varnothing \cap -(U_\varnothing)_1$ and note that $A_0 \in \U$ since $\U$ is subsymmetric and $U_\varnothing \in \U$. By definition of $\U$ we have $$(\U p)(\U q) \left(\langle p, q \rangle_{(-T)} \subseteq A_0\right)$$ and so we take any $p_0 \in \fin_{\pm k}$ such that $(\U q)\left( \langle p_0, q \rangle_{(-T)} \subseteq A_0\right)$; in particular $p_0 \in A_0$ by definition of $\langle p_0, q \rangle_{(-T)}$. We then take $A_1$ to be the intersection of the set $\{q \in \fin_{\pm k} : \langle p_0, q \rangle_{(-T)} \subseteq A_0\}$ with $$\bigcap\left\{ U_t \cap -(U_t)_1 : t\in U \text{ and} \, \supp \bigcup t \subseteq \supp p_0\right\}.$$ Note that $A_0 \supseteq A_1$ and $A_1 \in \U$ since there are only finitely many $t \in U$ satisfying $\supp \bigcup t \subseteq \supp p_0$, and since each $U_t \cap -(U_t)_1 \in \U$ using the fact that $\U$ is subsymmetric.

Now suppose $A_0, \dots, A_n$ and $p_0, \dots, p_{n-1}$ have been constructed. Since $\U$ is cofinite, pick any $p_n \in \fin_{\pm k}$ such that $p_n > p_{n-1}$ and $(\U q)\left( \langle p_n, q \rangle_{(-T)} \subseteq A_n\right)$; in particular $p_n \in A_n$. Then take $A_{n+1}$ to be the intersection of the set $\{q \in \fin_{\pm k} : \langle p_n, q \rangle_{(-T)} \subseteq A_n\}$ with $$\bigcap\left\{ U_t \cap -(U_t)_1 :t \in U \text{ and} \, \supp \bigcup t \subseteq  \bigcup_{i<n+1}\supp p_i\right\}.$$ As before, we have $A_{n+1} \in \U$ and $A_n \supseteq A_{n+1}$. This completes the induction.

To check that $P$ is the desired block sequence, we prove the following properties:
\begin{enumerate}
	\item[(4)] $\langle p_m, \dots, p_n\rangle_{(-T)} \subseteq A_m$ for all $m \leq n < \omega$.
	\item[(5)] If $Q = (q_n)_{n < \omega} \leq_{(-T)} P$, then $(q_0, \dots, q_m) \in U$ for all $m < \omega$.
\end{enumerate}

We check (4) by downward induction on $m \leq n$ for $n < \omega$ fixed. The case $m = n$ follows from (1), while the case $m = n -1 $ follows using (1) and (2) to obtain $\langle p_{n-1}, p_n \rangle_{(-T)} \subseteq A_{n-1}$. Now suppose inductively that (4) holds for some $m \leq n$; we aim to show $\langle p_{m-1}, p_m, \dots, p_n \rangle_{(-T)} \subseteq A_{m-1}$. Take any $$q = \sum_{i = m-1}^n (-T)^{j_i}(p_i)$$ with $j_{m-1}, \dots, j_n \in \{0, \dots, k\}$ and $\min j_i = 0$. We consider two cases: Suppose first that there is $i > m-1$ such that $j_i = 0$. Then $$q' := \sum_{i=m}^n (-T)^{j_i}(p_i) \in \langle p_m, \dots, p_n\rangle_{(-T)} \subseteq A_m$$ where the inclusion comes from the inductive hypothesis. Then $q' \in A_m$ and so $$q \in \langle p_{m-1}, q \rangle_{(-T)} \subseteq A_{m-1}$$ by (2). Now suppose $j_i > 0$ for each $i > m-1$ (so that, in particular, $j_{m-1} = 0$). Let $l := \min\{j_m, \dots, j_n\} > 0$ and write $$q = p_{m-1} + (-T)^l \left(\sum_{i=m}^n (-T)^{j_i - l}(p_i)\right).$$ By the inductive hypothesis we have $$q'' := \sum_{i=m}^n (-T)^{j_i - l}(p_i) \in \langle p_m, \dots, p_n \rangle_{(-T)} \subseteq A_m,$$ and so $q \in \langle p_{m-1}, q'' \rangle_{(-T)} \subseteq A_{m-1}$ by (2). This completes the proof of (4).

Let $Q$ be as in the statement of (5) and fix $q = (q_0, \dots, q_m)$. We prove $q \in U$ by induction on $m < \omega$. If $m = 0$ then $q = (q_0)$ and by definition of $Q$ we can write $$q_0 = \sum_{i < l} (-T)^{j_i}(p_{n_i})$$ for some $l < \omega, n_0 < \dots < n_{l-1} < \omega$ and $j_i \in \{0, \dots, k\}$ with $\min j_i = 0$. Then $q_0 \in \langle p_{n_0}, \dots, p_{n_{l-1}} \rangle_{(-T)}$ and so by (4) we have $$q_0 \in A_{n_0} \subseteq A_0 \subseteq U_\varnothing$$ where we use the definition of $A_0$ above. Thus $q = (q_0) \in U$. Now suppose $m > 0$ and write $t := (q_0, \dots, q_{m-1})$ so that $q = (t, q_m)$ and $t \in U$ by the inductive assumption. Again, by definition of $Q$ we can write $$q_m = \sum_{i < l} (-T)^{j_i}(p_{n_i})$$ for some $l < \omega, n_0 < \dots < n_{l-1} < \omega$ and $j_i \in \{0, \dots, k\}$ with $\min j_i = 0$. Then $q_m \in \langle p_{n_0}, \dots, p_{n_{l-1}} \rangle_{(-T)}$ and so by (4) we have $q_m \in A_{n_0}$. Since $q_{m-1} < q_m$ it must be the case that $$\supp \bigcup t \subseteq  \bigcup_{i < n_0} \supp p_i.$$ Then by (3) we obtain $q_m \in U_t$ and so $q = (t, q_m) \in U$. This finishes the inductive proof of (5) and hence the proof of the lemma is complete. \end{proof}

In what follows, we will only need the following corollary of the above proof.

\begin{cor}\label{cor2} For every $\U$-tree $U$ with stem $\varnothing$ there is $P = (p_n)_{n < \omega} \in \fin_{\pm k}^{[\infty]}$ together with a sequence $A_0 \supseteq A_1 \supseteq \dots$ of subsets of $\fin_{\pm k}$ such that:
	\begin{enumerate}
		\item $A_n \subseteq U_t \cap -(U_t)_1$ for every $t \in U$ such that $\supp \bigcup t \subseteq \bigcup_{i < n} \supp p_i$,
		\item $\langle p_m, \dots, p_n\rangle_{(-T)} \subseteq A_m$ for all $m \leq n < \omega$.
	\end{enumerate}
\end{cor}

Recall that for a block sequence $P = (p_n)_{n < \omega}$ in $\fin_{\pm k}$, $\langle P \rangle_{\pm k}^{[\infty]}$ denotes the set of all infinite block subsequences of $P$ in $\fin_{\pm k}$. We then have the following key lemma which makes use of the $S$-closed $\U$-trees defined in the previous section.

\begin{lem}\label{lem4} Let $U$ be an $S$-closed $\U$-tree with $\stem(U) = \varnothing$. Then there is an infinite block sequence $P = (p_n)_{n < \omega}$ in $\fin_{\pm k}$ such that $\langle P\rangle_{\pm k}^{[\infty]} \subseteq ([U])_3$. \end{lem}

\begin{proof} Find an infinite block sequence $P$ as in Corollary \ref{cor2}. We claim that $P$ satisfies the conclusion of the lemma. To see this, fix an infinite block subsequence $Q = (q_n)_{n < \omega}$ of $P$. For convenience, we fix some notation: For each $n < \omega$ let $I_n$ be the smallest set of non-negative integers such that $$q_n \in \langle p_i : i \in I_n \rangle_{\pm k}.$$ Notice that since $Q$ is a block subsequence of $P$ we have $\max I_n < \min I_m$ whenever $n < m$.

We will find a block sequence $Q' = (q_n')_{n <\omega} \in [U]$ such that $||q_n - q_n'|| \leq 3$ and $\supp q_n' \subseteq \supp q_n$ for all $n < \omega$. We define $Q'$ recursively as follows. For $n = 0$, write $$q_0 = \sum_{i \in I_0} \varep_i T^{j_i}(p_i)$$ for some (necessarily unique) $\varep_i \in \{\pm 1\}$ and $j_i < k$ such that $\min j_i = 0$. We consider the following two cases:

\begin{case1} There is $i \in I_0$ such that $\varep_i = +1$ and $j_i = 0$.\end{case1}
\noindent For each $i \in I_0$, set $r_i := \varep_i T^{j_i}(p_i)$ for convenience. We consider the following two subcases:
\begin{enumerate}[(a)]
	\item $\varep_i = +1$ and $j_i$ is even, or $\varep_i = -1$ and $j_i$ is odd. In either case, set $r_i' := r_i$ and note that $r_i' = (-T)^{j_i}(p_i)$.
	\item $\varep_i = +1$ and $j_i$ is odd, or $\varep_i = -1$ and $j_i$ is even. In either case, set $r_i' := T(r_i)$ and note that $r_i' = (-T)^{j_i + 1}(p_i)$.
\end{enumerate}
We then set $$q_0' := \sum_{i \in I_0} r_i'.$$ Note that $\supp q_0' \subseteq \supp q_0$ and $q_0' \in \langle p_i : i \in I_0 \rangle_{(-T)}$ by the assumption given by Case 1. Since $||r_i - r_i'|| \leq 1$ for all $i \in I_0$ we have $||q_0 - q_0'|| \leq 1$. Furthermore, by Corollary \ref{cor2} we have $$\langle p_i : i \in I_0 \rangle_{(-T)} \subseteq A_{\min I_0}$$ (using the notation of Corollary \ref{cor2}) and so $q_0' \in U_t$ for every $t \in U$ such that $$\supp \bigcup t \subseteq \bigcup_{i < \min I_0} \supp p_i.$$ In particular, $q_0' \in U_\varnothing$ and so $(q_0') \in U$.

\begin{case2} For every $i \in I_0$, if $j_i = 0$ then $\varep_i = -1$.\end{case2}
\noindent Apply Case 1 to $-q_0$ to obtain $r \in \langle p_i : i \in I_0 \rangle_{(-T)}$ such that $||(-q_0) - r || \leq 1$ and $\supp r \subseteq \supp (-q_0)$. By Corollary \ref{cor2} we have  $$\langle p_i : i \in I_0 \rangle_{(-T)} \subseteq A_{\min I_0}$$  and so $r \in U_t \cap -(U_t)_1$ for every $t \in U$ such that $$\supp \bigcup t \subseteq \bigcup_{i < \min I_0} \supp p_i.$$ In particular, $-r \in (U_\varnothing)_1$ and so there is $r' \in U_\varnothing$ such that $||(-r) -r' || \leq 1$. Since $U$ is $S$-closed, we have $(S(r')) \in U$ and so we set $q_0' := S(r')$. Note that by definition of $S$ we have $$\supp q_0' \subseteq \supp (-r) = \supp r \subseteq \supp q_0.$$ Furthermore, using the fact that $|| r' - S(r')|| \leq 1$ we have $$||q_0 - q_0'|| \leq ||q_0 - (-r)|| + ||(-r) - r'|| + ||r' - S(r')|| \leq 3$$ and so $q_0'$ satisfies our requirements.

Now assume $n > 0$ and suppose we have defined $q_0', \dots, q_{n-1}' \in \fin_{\pm k}$ such that $s := (q_0', \dots, q_{n-1}') \in U$, $||q_i - q_i'|| \leq 3$ and $\supp q_i' \subseteq \supp q_i$ for all $i < n$. Write $$q_n = \sum_{i \in I_n} \varep_i T^{j_i}(p_i)$$ for some $\varep_i \in \{\pm 1\}$ and $j_i < k$ such that $\min j_i = 0$. Note that since $$\supp q_i' \subseteq \supp q_i \subseteq \bigcup_{j \in I_i} \supp p_i,$$ we must have $$\supp \bigcup s \subseteq \bigcup_{i < \min I_n} \supp p_i.$$ As in the base case of the induction, we consider the following two cases:

\begin{case1} There is $i \in I_n$ such that $\varep_i = +1$ and $j_i = 0$.\end{case1}
\noindent For each $i \in I_n$, set $r_i := \varep_i T^{j_i}(p_i)$ for convenience. We consider the following two subcases:
\begin{enumerate}[(a)]
	\item $\varep_i = +1$ and $j_i$ is even, or $\varep_i = -1$ and $j_i$ is odd. In either case, set $r_i' := r_i$ and note that $r_i' = (-T)^{j_i}(p_i)$.
	\item $\varep_i = +1$ and $j_i$ is odd, or $\varep_i = -1$ and $j_i$ is even. In either case, set $r_i' := T(r_i)$ and note that $r_i' = (-T)^{j_i + 1}(p_i)$.
\end{enumerate}
We then set $$q_n' := \sum_{i \in I_n} r_i'.$$ As before, we have $\supp q_n' \subseteq \supp q_n$ and $||q_n - q_n'|| \leq 1$. Furthermore, we have $$\langle p_i : i \in I_n \rangle_{(-T)} \subseteq A_{\min I_n}$$ and so $q_n' \in U_t$ for every $t \in U$ such that $$\supp \bigcup t \subseteq \bigcup_{i < \min I_n} \supp p_i.$$ In particular, $q_n' \in U_s$ and so $(s, q_n') \in U$.

\begin{case2} For every $i \in I_n$, if $j_i = 0$ then $\varep_i = -1$.\end{case2}
\noindent Apply Case 1 to $-q_n$ to obtain $r \in \langle p_i : i \in I_n \rangle_{(-T)}$ such that $||(-q_n) - r || \leq 1$ and $\supp r \subseteq \supp (-q_n)$. As before, $r \in U_t \cap -(U_t)_1$ for every $t \in U$ such that $$\supp \bigcup t \subseteq \bigcup_{i < \min I_n} \supp p_i.$$ In particular, $-r \in (U_s)_1$ and so there is $r' \in U_s$ such that $||(-r) -r' || \leq 1$. Since $U$ is $S$-closed, we have $(s, S(r')) \in U$ and so we set $q_n' := S(r')$. As before, we check that $q_n'$ satisfies our requirements. This completes the inductive construction of $Q'$. It is clear from the above construction that $Q' \in [U]$ and $||q_n - q_n'|| \leq 3$ for all $n < \omega$ and so $Q \in ([U])_3$. \end{proof}

To finish the proof of Theorem \ref{mainthm} we will need the following mapping which was originally used in \cite{K} to give an alternate proof of Gowers' theorem. Given $m \in \N$, let $\Phi_m : \fin_{\pm 2m} \rightarrow \fin_{\pm m}$ be defined by setting, for $p \in \fin_{\pm 2m}$ and $n < \omega$,
\[ \Phi_m (p)(n) := \begin{cases} 
      \frac{p(n)}{2} & \text{ if $p(n)$ is even}, \\
      \frac{p(n)-1}{2} & \text{ if $p(n)>0$ and $p(n)$ is odd}, \\
      \frac{p(n)+1}{2} & \text{ if $p(n)<0$ and $p(n)$ is odd}.
   \end{cases}
\]
The following lemma is easy to check.
\begin{lem}\label{lem5} For each $m \in \N$, the mapping $\Phi_m$ has the following properties:
	\begin{enumerate}[(i)]
	\item $\Phi_m$ is a surjective homomorphism of partial semigroups which, in addition, satisfies $\Phi_m(-p) = -\Phi_m(p)$ for every $p \in \fin_{\pm 2m}$.
	\item For every $p_0 < p_1 \in \fin_{\pm 2m}$ and every $j_0, j_1 < k + 1$ with $\min\{j_0, j_1\} = 0$, we have $$\Phi_m\left(T^{2j_0}(p_0) + T^{2j_1}(p_1)\right) = T^{j_0}(\Phi_m(p_0)) + T^{j_1}(\Phi_m(p_1)).$$
	\item For every $p_0, p_1 \in \fin_{\pm 2m}$ and every $l < \omega$, we have $$||p_0 - p_1|| \leq 2l \implies ||\Phi_m(p_0) - \Phi_m(p_1)||\leq l.$$
\end{enumerate}
\end{lem}

Now, for $k \in \N$ fixed as in the previous sections, let $\Psi : \fin_{\pm 4k} \rightarrow \fin_{\pm k}$ be given by $\Psi := \Phi_k \circ \Phi_{2k}.$ Using the properties listed in Lemma \ref{lem5} it is easy to verify that $\Psi$ is a surjective homomorphism which satisfies:
\begin{enumerate}[(a)]
	\item For every $p_0 < p_1 \in \fin_{\pm 4k}$ and every $j_0, j_1 < k + 1$ with $\min\{j_0, j_1\} = 0$, we have $$\Psi\left(T^{4j_0}(p_0) + T^{4j_1}(p_1)\right) = T^{j_0}(\Psi(p_0)) + T^{j_1}(\Psi(p_1)).$$
	\item For every $p_0, p_1 \in \fin_{\pm 4k}$, if $||p_0 - p_1|| \leq 4$ then $||\Psi(p_0) - \Psi(p_1)||\leq 1$.
\end{enumerate}
We extend $\Psi$ to $\fin_{\pm 4k}^{[\infty]}$ by setting $$\Psi((p_n)_{n < \omega}) := (\Psi(p_n))_{n < \omega}.$$ It is straightforward to check that $\Psi$ is continuous with respect to the usual metrizable topologies. Furthermore, note that if $P$ and $P'$ are two block sequences in $\fin_{\pm 4k}$ which satisfy $||P-P'|| \leq 4$, then $||\Psi(P) - \Psi(P')|| \leq 1$. We are now ready to finish the proof of the main theorem.

\begin{proof}[Proof of Theorem \ref{mainthm}]
Let $c : \fin_{\pm k}^{[\infty]} \rightarrow r$ be Souslin measurable. We define a colouring $\widetilde{c} : \fin_{\pm 4k}^{[\infty]} \rightarrow r$ by setting $\widetilde{c} := c \circ \Psi$. Then $\widetilde{c}$ is Souslin measurable since the collection $$\{\X \subseteq \fin_{\pm k}^{[\infty]} : \Psi^{-1}(\X) \subseteq \fin_{\pm 4k}^{[\infty]} \text{ is Souslin measurable}\}$$ is a field of subsets of $\fin_{\pm k}^{[\infty]}$ which contains the open sets (by continuity) and is closed under the Souslin operation, and hence contains the Souslin measurable subsets of $\fin_{\pm k}^{[\infty]}$. By Corollary \ref{cor1} there are $i < r$ and a $\U$-tree $V$ with stem $\varnothing$ such that $[V] \subseteq \widetilde{c}^{-1}\{i\}$. Applying Lemma \ref{lem2}, find an $S$-closed $\U$-tree $U$ such that $[U] \subseteq ([V])_1$; in particular we get $$[U] \subseteq \left(\widetilde{c}^{-1}\{i\}\right)_1.$$ Since $\U$ is $S$-closed, by Lemma \ref{lem4} we can find an infinite block sequence $\widetilde{P} = (\widetilde{p_n})_{n < \omega}$ in $\fin_{\pm 4k}$ such that $\langle \widetilde{P} \rangle_{\pm 4k}^{[\infty]} \subseteq ([U])_3$ and hence $$\langle \widetilde{P}\rangle_{\pm 4k}^{[\infty]} \subseteq \left(\widetilde{c}^{-1}\{i\}\right)_4.$$

Let $P := \Psi(\widetilde{P}) \in \fin_{\pm k}^{[\infty]}$ and set $p_n := \Psi(\widetilde{p_n})$ for each $n<\omega$. We claim that $P$ satisfies $$\langle P\rangle_{\pm k}^{[\infty]} \subseteq \left(c^{-1}\{i\}\right)_1.$$ Indeed, if $Q = (q_n)_{n < \omega} \in \fin_{\pm k}^{[\infty]}$ is an infinite block subsequence of $P$, then for each $n < \omega$ we have $$q_n = \sum_{i < m} \varep_i T^{j_i}(p_{n_i})$$ for some $\varep_i \in \{\pm 1\}, n_0 < \dots < n_{m-1}$ and $j_i < k$ such that $\min j_i = 0$. Then using property (a) of $\Psi$ listed above we see that $q_n = \Psi(\widetilde{q_n})$, where $$\widetilde{q_n} := \sum_{i < m} \varep_i T^{4j_i}(\widetilde{p_{n_i}}) \in \langle \widetilde{P} \rangle_{\pm k}$$ and so, setting $\widetilde{Q} := (\widetilde{q_n})_{n<\omega}$, we see that $Q = \Psi(\widetilde{Q})$. Since $\widetilde{Q}$ is a block subsequence of $\widetilde{P}$, by our choice of $\widetilde{P}$ we can find $Q' \in \widetilde{c}^{-1}\{i\}$ such that $||\widetilde{Q} - Q' || \leq 4$. Then, as observed above, property (b) of $\Psi$ implies $||\Psi(\widetilde{Q}) - \Psi(Q')|| \leq 1$. Since $$ i = \widetilde{c}(Q') = c(\Psi(Q'))$$ we obtain $\Psi(Q') \in c^{-1}\{i\}$ and so $Q \in \left(c^{-1}\{i\}\right)_1$ as required.
\end{proof}

In fact, we can do a bit better: Given an infinite block sequence $P$ in $\fin_{\pm k}$, the proof of Lemma \ref{lem1} (from either \cite{AT} or \cite{K}) can be adapted to show the existence of an ultrafilter $\U$ on the partial semigroup $\langle P \rangle_{\pm k}$ which has the properties listed in Lemma \ref{lem1}. One can then develop the theory of $\U$-trees on $\langle P \rangle_{\pm k}^{[<\infty]}$ and prove a corresponding analogue of Corollary \ref{cor1}. By equipping $\langle P \rangle_{\pm k}^{[\infty]}$ with its natural analogue of the metrizable topology and replacing $\fin_{\pm k}^{[\infty]}$ with $\langle P \rangle_{\pm k}^{[\infty]}$ in the proof of the main result, we obtain the following relativized version of Theorem \ref{mainthm}.

\begin{thm}\label{mainrel} For every $k, r \in \N$, every infinite block sequence $P$ in $\fin_{\pm k}$ and every Souslin measurable $c : \fin_{\pm k}^{[\infty]} \rightarrow r$ there are $i < r$ and an infinite block sequence $Q \leq P $ such that $$\langle Q \rangle_{\pm k}^{[\infty]} \subseteq \left(c^{-1}\{i\}\right)_1.$$ \end{thm}

The previous result can be used to ``diagonalize'' Theorem \ref{mainthm} as follows. First note that, for each $j < k \in \N$, the $j^{\mathrm{th}}$ iterate of the tetris operation $T^{(j)} : \fin_{\pm k} \rightarrow \fin_{\pm (k-j)}$ can be extended to $T^{(j)} : \fin_{\pm k}^{[\infty]} \rightarrow \fin_{\pm (k-j)}^{[\infty]}$ by setting $$T^{(j)}((p_n)_{n < \omega}) := (T^{(j)}(p_n))_{n < \omega}.$$ We then have the following:

\begin{cor} For every $k, r \in \N$ and every Souslin measurable (with respect to the disjoint union topology) colouring $$c : \bigcup_{j = 1}^k \fin_{\pm j}^{[\infty]} \rightarrow r$$ there are $i_1, \dots, i_k < r$ and $P \in \fin_{\pm k}^{[\infty]}$ such that $$\langle T^{(k-j)} (P) \rangle_{\pm j}^{[\infty]} \subseteq \left(c^{-1}\{i_j\}\right)_1$$ for each $j = 1, \dots, k$. \end{cor}

\begin{proof} Note that each canonical inclusion $$\iota_j : \fin_{\pm j}^{[\infty]} \rightarrow \bigcup_{j = 1}^k \fin_{\pm j}^{[\infty]}$$ is continuous and so, as in the proof of Theorem \ref{mainthm}, each Souslin measurable colouring of the union induces a Souslin measurable colouring of $\fin_{\pm j}^{[\infty]}$ by composing with $\iota_j$, for each $j \in \{1,\dots,k\}$. Thus by Theorem \ref{mainthm} we can find $P_1 \in \fin_{\pm 1}^{[\infty]}$ and $i_1 < r$ such that $\langle P_1 \rangle_{\pm 1}^{[\infty]} \subseteq (c^{-1}\{i_1\})_1.$ Take any $Q_2 \in \fin_{\pm 2}^{[\infty]}$ such that $T(Q_2) = P_1$ and apply Theorem \ref{mainrel} to $Q_2$ to obtain $P_2 \leq Q_2$ and $i_2 < r$ such that $\langle P_2 \rangle_{\pm 2}^{[\infty]} \subseteq (c^{-1}\{i_2\})_1.$ Continue inductively to obtain $P_j \leq Q_j \in \fin_{\pm j}^{[\infty]}$ and $i_j < r$, for $j = 2, \dots, k$, such that $T(Q_j) = P_{j-1}$ and $\langle P_j \rangle_{\pm j}^{[\infty]} \subseteq (c^{-1}\{i_j\})_1.$

We claim that setting $P := P_k$ works. Indeed, for a fixed $j = 1, \dots, k$ we have $T^{(k-j)}(P) \leq P_j$ by construction (and using the general fact that $T(P) \leq T(Q)$ whenever $P \leq Q$) and so the desired conclusion follows from the choice of $P_j$.
\end{proof}

We conclude with a proof of the multi-dimensional version of Theorem \ref{mainthm}. Recall that, for $d \in \N$, $\fin_{\pm k}^{[d]}$ denotes the set of all block sequences in $\fin_{\pm k}$ of length $d$. Given an infinite block sequence $P$ let $\langle P \rangle_{\pm k}^{[d]}$ be the set of all $Q = (q_n)_{n < d} \in \fin_{\pm k}^{[d]}$ such that $q_n \in \langle P \rangle_{\pm k}$ for each $n < d$.

\begin{cor}\label{mainmulti} For every $k, d, r \in \N$ and every colouring $c : \fin_{\pm k}^{[d]} \rightarrow r$ there are $i < r$ and an infinite block sequence $P \in \fin_{\pm k}^{[\infty]}$ such that $$\langle P \rangle_{\pm k}^{[d]} \subseteq \left(c^{-1}\{i\}\right)_1.$$ \end{cor}

\begin{proof} Given a colouring $c$ as above, let $\widetilde{c} : \fin_{\pm k}^{[\infty]} \rightarrow r$ be given by $$\tilde{c}((p_n)_{n<\omega}) := c((p_n)_{n < d}).$$ Then $\widetilde{c}$ is continuous and hence Souslin measurable since for each $i < r$ we have $$\widetilde{c}^{-1}\{i\} = \bigcup \left\{[Q] : Q \in \fin_{\pm k}^{[d]} \, \cap \, c^{-1}\{i\}\right\}$$ (recall that $[Q]$ denotes the basic open set consisting of all infinite block sequences which begin with $Q$). By Theorem \ref{mainthm} there are $i < r$ and $P \in \fin_{\pm k}^{[\infty]}$ such that $$\langle P \rangle_{\pm k}^{[\infty]} \subseteq \left(\widetilde{c}^{-1}\{i\}\right)_1.$$ Given $Q = (q_n)_{n < d} \in \langle P \rangle_{\pm k}^{[d]}$ extend $Q$ arbitrarily to any $\widetilde{Q} \in \langle P \rangle_{\pm k}^{[\infty]} \cap [Q]$. By choice of $P$ there is $Q' = (q_n')_{n < \omega} \in \widetilde{c}^{-1}\{i\}$ such that $||q_n - q_n'|| \leq 1$ for all $n < \omega$. Then $c((q_n')_{n < d})= i$ and so $Q \in (c^{-1}\{i\})_1$.
\end{proof}

\bibliographystyle{amsplain}
\bibliography{bibliography}

\end{document}